\documentclass[12pt]{amsart}
\usepackage{amsmath,amssymb,amsfonts,amsthm,amsopn}
\usepackage{graphics}
\usepackage{srcltx}

\newcommand{\tfs}{time-frequency shift}

\newtheorem{tm}{Theorem}[section]
\newtheorem{lemma}[tm]{Lemma}

\newtheorem{Remark}[tm]{Remark}

\newtheorem{theorem}{Theorem}[section]
\newtheorem{corollary}[theorem]{Corollary}
\newtheorem{definition}[theorem]{Definition}

\newtheorem{proposition}[theorem]{Proposition}

\newcommand{\beqa}{\begin{eqnarray*}}
\newcommand{\eeqa}{\end{eqnarray*}}
\newcommand{\tpsdo}{$\tau$-pseudo\-differential operator}

\newcommand{\field}[1]{\mathbb{#1}}
\newcommand{\bR}{\field{R}}        
\newcommand{\bN}{\field{N}}        
\newcommand{\bZ}{\field{Z}}        
\newcommand{\bC}{\field{C}}        
\newcommand{\opt }{\mathrm{Op}_{\tau}}

\def\la{\lambda}

 \def\cF{\mathcal{F}}              
 \def\cS{\mathcal{S}}
 \def\cD{\mathcal{D}}
 
 \def\cB{\mathcal{B}}

 \def\cC{\mathcal{C}}

 \def\cN{\mathcal{N}}
 
 \def\cT{\mathcal{T}}

\def\a{\aleph}

\def\rd{\bR^d}

\def\rdd{{\bR^{2d}}}

\def\lrd{L^2(\rd)}

\def\zd{\bZ^d}

\def\intrd{\int_{\rd}}
\def\intrdd{\int_{\rdd}}

\def\R{\right)}

\def\<{\left<}
\def\>{\right>}

\def\mv1{M_v^1}

\def\Lmpq{L_w^{p,q}}

\def\Mmpq{M_w^{p,q}}
\def\phas{(x,\xi )}

\def\o{\xi}
\def\a{\alpha}

\def\ZZ{\mathbb{Z}}

\def\N{\mathbb{N}}
\def\R{\mathbb{R}}
\def\Ren{\mathbb{R}^d}
\def\Renn{\mathbb{R}^{2d}}

\def\Sn2{S_{2}(L^{2}(\Ren))}
\def\S1{S_{1}(L^{2}(\Ren))}
\def\sig00{\sigma_{0,0}}

\def\la{\langle}
\def\ra{\rangle}

\begin{document}

\title[]{On the local well-posedness of the nonlinear heat equation associated to the fractional  Hermite operator in  modulation spaces}

\author{Elena Cordero}
\address{Department of Mathematics,  University of Torino,
Via Carlo Alberto 10, 10123
Torino, Italy}
\email{elena.cordero@unito.it}

\keywords{Modulation spaces,  Gabor matrix, fractional  Hermite operators, Shubin and H\"{o}rmander classes}

\subjclass[2010]{42B35,35K05,35K55,35A01}

\date{}

\begin{abstract} In this note we consider  the nonlinear heat equation associated to the fractional  Hermite operator $H^\beta =(-\Delta+|x|^2)^\beta$, $0<\beta\leq 1$. We  show the local solvability of the related  Cauchy problem  in the framework of modulation spaces. The result is obtained by combining tools from microlocal and time-frequency analysis. As a byproduct, we compute the Gabor matrix of pseudodifferential operators  with symbols in the H\"{o}rmander class $S^m_{0,0}$, $m\in\bR$. 
\end{abstract}

\maketitle

\section{Introduction and results}

In this  note we study the  Cauchy problem for
 the nonlinear heat equation associated to the fractional  Hermite operator
%
\begin{equation}\label{cpw}
\begin{cases}
\partial_t u+H^\beta u=F(u)\\
u(0,x)=u_0(x)
\end{cases}
\end{equation}
with  $t\in [0,T]$, $T>0$, $x\in\R^d$, $H^\beta =(-\Delta+|x|^2)^\beta$, $0<\beta\leq 1$, $\Delta=\partial^2_{x_1}+\dots \partial^2_{x_d}$, $d\geq1$. $F$ is a scalar
 function on $\bC$, with
 $F(0)=0$. The solution  $u(t,x)$ is  a complex valued function of $(t,x)\in \R\times\rd$. We will consider
 the case in which $F$  is a real analytic function with an entire extension. \par
 Well-posedness of the heat equation has been studied by many authors, see e.g. \cite{nicola,fio2} and the many contributions by Wong, for instance \cite{Wong1,Wong2}, see also \cite{catana}.    In particular,  heat equations associated to  fractional Hermite operators were recently studied in  \cite{bhimani1}, for related results see also \cite{bhimNTI}. Hermite multipliers  are considered in \cite{bhimani2}, see also the textbook \cite{thangavelu}.  Recently the study of Cauchy problems in modulation spaces have been pursued by many authors, see the pioneering works \cite{benyi,benyi3}. Many deep results in this framework for nonlinear evolution equations have been obtained by B. Wang et al. in \cite{baoxiang3,baoxiang2} and are also available in the textbook \cite{Wangbook2011}.
 
 Following the spirit of \cite{CNwave,CZ}, we shall prove the local  existence and uniqueness of the solutions in modulation spaces to the Cauchy problem  \eqref{cpw}. The  key 
  arguments come from both microlocal and time-frequency analysis.  In fact, we shall rely on the results related to spectral theory of globally elliptic operators developed by Helffer \cite{helffer} to  understand the properties of the fractional Hermite operators $H^\beta$. Namely, they are pseudodifferential  operators with  Weyl symbols $a_\beta$ positive globally elliptic and in the Shubin classes $\Gamma^{2\beta}_1$ (see Definition \ref{defshubin} and the estimate \eqref{pge} below). 
 
The spectral decomposition of the Hermite operator  $H=-\Delta+|x|^2$ is given by $H=\sum_{k=0}^\infty (2k +d) P_k$, where  $P_k$ is  the orthogonal projection of $L^2(\rd)$
  onto the eigenspace corresponding to the eigenvalue $(2k + d)$. Namely, the range of the operator $P_k$ is the space spanned by the Hermite
  functions $\Phi_\alpha$ in $\rd$,  with $\alpha$ multi-index in $\bN^d$, such that $|\a|=k$.
  The solution to the homogeneous Cauchy problem \eqref{cpw} (i.e., $F=0$) 
  can be formally written in terms of the heat semigroup related to $H^\beta$
  \[
  e^{-tH^\beta}=\sum_{k=0}^{+\infty} e^{-t(2k+d)^\beta} P_k
  \]
  as 
  ${u}(t,x)=K_\beta(t)u_0=e^{-tH^\beta} u_0(x),\quad t\geq 0,\, x\in\rd.$  
  
 We shall prove that the propagator $K_\beta(t)=e^{-t H^\beta}$ can be represented as a pseudodifferential operator with Weyl symbol in the Shubin class $\Gamma^0_1$, with related semi-norms uniformly bounded with respect to the time variable $t\in [0,T]$, for any fixed $T>0$.   
 
 After that we shall leave the microlocal techniques to come to time-frequency analysis. We  perform a general study concerning 
 the boundedness of  Shubin $\tau$-pseudodifferential operators with symbols in the  H{\"o}rmander classes $S^m_{0,0}$ (for $\tau=1/2$ we recapture the Weyl case). The outcomes are contained in Theorem \ref{GBsm} below (see also the subsequent corollary and remark).
 
 The main tool here is to study the decay of  their related Gabor matrix representations, which we shall also control by the semi-norms in  $S^m_{0,0}$. We think that such result is valuable in and of itself.

We then use the special case of Weyl operators to study \eqref{cpw}.
The integral version of the
problem \eqref{cpw} has the
form
\begin{equation}\label{solop}
    u(t,\cdot)=K_\beta(t)u_0+\mathcal{B}F(u),
\end{equation}
where\begin{equation}\label{op2}
\mathcal{B}=\int _0^t K_\beta(t-\tau)\cdot d\tau.
\end{equation}
To show that the Cauchy Problem \eqref{cpw} has a unique solution, we use a variant of the contraction mapping theorem (see Proposition \ref{AIA} below).
 \par
  
As already mentioned, the function spaces used for our results are weighted modulation spaces $\Mmpq$, $1\leq p,q\leq\infty$, introduced by H. Feichtinger in 1983 \cite{F1} (then extended to $0<p,q\leq\infty$ in \cite{Galperin2004}). We refer the reader to Section $2$ for their definitions and main properties.

The local well-posedness results for modulation spaces read as follows:
\begin{theorem}\label{T1} Assume $s\geq 0$,
 $1\leq p<\infty$,
 $u_0\in {M}^{p,1}_s(\rd)$
and $$F(z)=\sum_{j,k=0}^\infty
c_{j,k} z^j \bar{z}^k,$$ an
entire real-analytic function
on $\bC$ with $F(0)=0$. For
every $R>0$, there exists
$T>0$ such that for every
$u_0$ in the ball $B_R$
of center $0$ and radius $R$
in ${M}^{p,1}_s(\rd)$
there exists a unique
solution $u\in
\cC^0([0,T];{M}^{p,1}_s(\rd))$
to \eqref{cpw}.
Furthermore, the map
$u_0\mapsto u$ from
$B_R$ to $
\cC^0([0,T];{M}^{p,1}_s(\rd))$
is Lipschitz continuous.  

For $p=\infty$ the result still holds if we replace $M^{\infty,1}_s(\rd)$ with the space $\mathcal{M}^{\infty,1}_s(\rd)$, the closure of the Schwartz class in the $M^{\infty,1}_s$-norm.
\end{theorem}

We actually do not know whether it is possible to obtain better results concerning the nonlinearity $F(u)=\lambda|u|^{2k} u$, $k\in\bN$,  we refer to the work \cite{bhimNTI} for a  discussion on the topic. 

The  tools employed follow the pattern of similar Cauchy
problems studied for other equations such as the Schr\"odinger,
wave and Klein-Gordon equations  \cite{benyi,benyi3,CNwave,CZ}.

To compare with other results in the literature, we observe that also in \cite{Wong1,Wong2}, and in \cite{catana} the authors use Wigner distributions and pseudodifferential operators as tools for their main results. In the latter paper the author gives a formula for
the one-parameter strongly continuous semigroup $e^{-t H^\beta}$ in terms of the Weyl transforms of a $L^2$-orthonormal basis made of generalized Hermite eigenfunctions. 
This is then used to obtain $L^2$-estimates for the solution of
the related  initial value problem with data in $L^p$ spaces, $1 \le p \leq\infty$. Here the approach uses similar ideas of joining microlocal and time-frequency analysis tools, but the spaces employed are different: we use modulation spaces, which are the most common ones in time-frequency analysis.

\section{Function spaces and preliminaries} 
We denote by $v$  a
continuous, positive,  submultiplicative  weight function on $\rd$, i.e., 
$ v(z_1+z_2)\leq v(z_1)v(z_2)$, for all $ z_1,z_2\in\Ren$.
We say that $w\in \mathcal{M}_v(\rd)$ if $w$ is a positive, continuous  weight function  on $\Ren$  {\it
	$v$-moderate}:
$w(z_1+z_2)\leq Cv(z_1)w(z_2)$  for all $z_1,z_2\in\Ren$ (or for all $z_1,z_2\in \zd$).
We will mainly work with polynomial weights of the type
\begin{equation}\label{vs}
v_s(z)=\la z\ra^s =(1+|z|^2)^{s/2},\quad s\in\bR,\quad z\in\rd\,\, (\mbox{or}\, \zd).
\end{equation}
Observe that,  for $s<0$, $v_s$ is $v_{|s|}$-moderate. Moreover, we limit to weights $w$ with at most polynomial growth, that is there exists $C>0, s>0$ such that
\begin{equation}\label{growth}
w(z)\leq C\la z\ra^s,\quad z\in\rdd.
\end{equation}\par 
We define $(w_1\otimes w_2)\phas:=w_1(x)w_2(\xi)$, for $w_1,w_2$ weights on $\rd$.

The main features of time-frequency analysis are $T_x$ and $M_\xi$, 
the so-called translation and
modulation operators, defined by $T_x
g(y)=g(y-x)$ and $M_\xi g(y)=e^{2\pi
	i\xi y}g(y)$. Let $g\in\cS(\rd)$ be a
non-zero window function in the Schwartz class  and consider
the  short-time Fourier
transform (STFT) $V_gf$ of a
function/tempered distribution $f$ in $\cS'(\rd)$ with
respect to the the window $g$:
\[
V_g f(x,\o)=\la f, M_{\o}T_xg\ra =\int e^{-2\pi i \o y}f(y)\overline{g(y-x)}\,dy,
\]
i.e.,  the  Fourier transform $\cF$
applied to $f\overline{T_xg}$. \par

For $z=(z_1,z_2)\in\rdd$, we call \emph{time-frequency shifts} the composition $$\pi(z)=M_{z_2}T_{z_1}.$$

{\bf Modulation Spaces.} For $1\leq p,q\leq \infty$ such spaces were introduced  by H. Feichtinger in \cite{F1} (see also their characterization in \cite{Birkbis}), then extended to $0<p,q\leq\infty$ by Y.V. Galperin and S. Samarah in \cite{Galperin2004}.
\begin{definition}\label{def2.4}
	Fix a non-zero window $g\in\cS(\rd)$, a weight $w\in\mathcal{M}_v$ and $0<p,q\leq \infty$. The modulation space $M^{p,q}_w(\rd)$ consists of all tempered distributions $f\in\cS'(\rd)$ such that the (quasi-)norm 
	\begin{equation}\label{norm-mod}
	\|f\|_{M^{p,q}_w}=\|V_gf\|_{L^{p,q}_w}=\left(\intrd\left(\intrd |V_g f \phas|^p w\phas^p dx  \right)^{\frac qp}d\o\right)^\frac1q 
	\end{equation}
	(with obvious changes with $p=\infty$ or $q=\infty)$ is finite. 
\end{definition}
For $1\leq p,q\leq \infty$ they  are Banach spaces, whose norm does not depend on the window $g$, in the sense that different window functions in $\cS(\rd)$ yield equivalent norms. Moreover, the window class $\cS(\rd)$ can be extended  to the modulation space  $M^{1,1}_v(\rd)$ (so-called Feichtinger algebra).

For shortness, we write $M^p_w(\rd)$ in place of $M^{p,p}_w(\rd)$,  $M^{p,q}(\rd)$ if $w\equiv 1$.  Moreover, for $w(x,\xi)=(1\otimes v_s)\phas$, we shall simply write, using the standard notation \cite{F1},
$$M^{p,q}_{1\otimes v_s}(\rd)=M^{p,q}_s(\rd).$$

In our study, we will apply Minkowski's integral inequality to study the operator $\mathcal{B}$ in \eqref{op2}. Such inequalities do not hold whenever the indices $p<1$ or $q<1$, hence we shall limit ourselves to the cases $1\leq p,q\leq\infty$.

We do not know whether the local well-posedness is still valid in the quasi-Banach setting.

Recall that  for $1\leq p,q\leq \infty$, $w\in\mathcal{M}_v$ and $g\in M^1_v(\rd)$, the norm 	$\|V_gf \|_{L^{p,q}_w}$ is an	equivalent norm for $M^{p,q}_w(\Ren)$ \cite[Thm.~11.3.7]{grochenig}). In other words, 
 given any $g \in M^1_v
(\rd )$ and $f\in\Mmpq $ we have the norm equivalence
\begin{equation}\label{normwind}
\|f\|_{\Mmpq } \asymp \|V_{g}f \|_{\Lmpq }.
\end{equation}
For this work we will use the inversion formula for
the STFT (see  \cite[Proposition 11.3.2]{grochenig}): assume $g\in M^{1}_v(\rd)\setminus\{0\}$,
$f\in M^{p,q}_w(\rd)$, with $w\in\mathcal{M}_v$, then
\begin{equation}\label{invformula}
f=\frac1{\|g\|_2^2}\int_{\R^{2d}} V_g f(z) \pi (z)  g\, dz \, ,
\end{equation}
and the  equality holds in $M^{p,q}_w(\rd)$.\par


We also recall their inclusion relations:
\begin{equation}\label{inclmod}{M}^{p_1,q_1}_w\hookrightarrow
{M}^{p_2,q_2}_w, \quad \mbox{if}\,\, p_1\leq p_2,\,\, q_1\leq
q_2.
\end{equation}

Other properties and more general definitions of modulation spaces can now  be found in  textbooks \cite{CR,grochenig}.

\subsection{ Shubin classes and symbols of  the operators $H^\beta$ and $e^{-t H^\beta}$}

Let us first recall the definition of Shubin classes (Shubin \cite[Definition 23.1]{shubin}):
\begin{definition}\label{defshubin}
	Let $m\in\mathbb{R}$. The symbol class $\Gamma_1%
	^{m}(\mathbb{R}^{2d})$ consists of all complex functions $a\in C^{\infty
	}(\mathbb{R}^{2d})$ such that for every $\alpha\in\mathbb{N}^{2d}$ there
	exists a constant $C_{\alpha}\geq0$ with
	\begin{equation}
	|\partial_{z}^{\alpha}a(z)|\leq C_{\alpha}\left\langle z\right\rangle
	^{m-|\alpha|},\quad z\in\mathbb{R}^{2d}. \label{est1}%
	\end{equation}	
\end{definition}
It immediately follows from this definition that if $a\in\Gamma_{1}%
^{m}(\mathbb{R}^{2n})$ and $\alpha\in\mathbb{N}^{2n}$ then $\partial
_{z}^{\alpha}a\in\Gamma_{1}^{m-|\alpha|}(\mathbb{R}^{2n}).$

Obviously $\Gamma_{1}^{m}(\mathbb{R}^{2n})$ is a complex vector space for
the usual operations of addition and multiplication by complex numbers, and we
have
\begin{equation}
\Gamma_{1}^{-\infty}(\mathbb{R}^{2d})=\bigcap\nolimits_{m\in\mathbb{R}%
}\Gamma_{1}^{m}(\mathbb{R}^{2d})=\mathcal{S}(\mathbb{R}^{2d}).
\label{gammaminf}%
\end{equation}

The notion of asymptotic expansion of a symbol $a\in\Gamma
_{1}^{m}(\mathbb{R}^{2d})$ (cf. \cite{shubin}, Definition 23.2) reads as follows.
\begin{definition}
	\label{23.2}Let $(a_{j})_{j}$ be a sequence of symbols $a_{j}\in\Gamma_{1
	}^{m_{j}}(\mathbb{R}^{2d})$ such that \\$\lim_{j\rightarrow+\infty}%
	m_{j}\rightarrow-\infty$. Let $a$ $\in C^{\infty}(\mathbb{R}^{2d})$. If for
	every integer $r\geq2$ we have%
	\begin{equation}
	a-\sum_{j=0}^{r-1}a_{j}\in\Gamma_{1}^{\overline{m}_{r}}(\mathbb{R}^{2d})
	\label{23.4}%
	\end{equation}
	where $\overline{m}_{r}=\max_{j\geq r}m_{j}$ we will write $a\thicksim
	\sum_{j=0}^{\infty}a_{j}$ and call this relation an asymptotic expansion of
	the symbol $a$.
\end{definition}

The interest of the asymptotic expansion comes from the fact that every
sequence of symbols $(a_{j})_{j}$ with $a_{j}\in\Gamma_{1}^{m_{j}%
}(\mathbb{R}^{2d})$, the degrees $m_{j}$ being strictly decreasing and such
that $m_{j}\rightarrow-\infty$ determines a symbol in some $\Gamma_{1}%
^{m}(\mathbb{R}^{2d})$, that symbol being unique up to an element of
$\mathcal{S}(\mathbb{R}^{2d})$.

The symbol of the Hermite operator (or  harmonic oscillator) $H(z)=(|x|^{2}
+4\pi^2|\xi|^{2})$ obviously belongs to $\Gamma_{1}^{2}(\mathbb{R}^{2d})$.

From  \cite[Theorem 1.11.1]{helffer} we infer that the fractional power $H^\beta$,  $0<\beta<1$,  can be written as a Weyl pseudodifferential operator having real symbol $a_\beta $ in the Shubin class $\Gamma^{2\beta}_1$   and positive globally elliptic. Recall that  a symbol $a_\beta$  is positive globally elliptic if  there exist $C>0$ and  $R>0$ such that 
\begin{equation}\label{pge}
a_\beta(z)\geq C\la z \ra, \quad |z|\geq R.
\end{equation}

Thanks to the properties of $H^\beta$ above, we can exploit a  result  by Nicola and Rodino in \cite[Theorem 4.5.1]{NR}  to prove  that the operator $e^{-tH^\beta}$ is a pseudodifferential operator with Weyl symbol in the Shubin class $\Gamma_1^0$, with uniform estimates  with respect to $t\in [0,T]$, for any fixed $T>0$.

For this purpose, we use  the above theorem in the following  setting: $$\Phi(z)=\Psi(z)=\la z\ra,\quad  h(z)=\Phi(z)^{-1}\Psi(z)^{-1}=\la z\ra ^{-2};$$ 
moreover we choose the parameters $l=N=0$ and $J=1$.
 If we consider the asymptotic expansion $a_\beta\sim \sum_{0}^\infty a_{\beta,j}$ as in Definition \ref{23.2}, together with the ellipticity condition  $a_{\beta,0}(z) \gtrsim \la z\ra ^{2\beta}$, $|z|\geq R$,  then Theorem $4.5.1.$ guarantees that the operator $e^{-tH^\beta}$ is a pseudodifferential operator with Weyl symbol $b(t,z)$ satisfying, for every $k\in\bN$,  $T>0$, the estimate
\begin{equation}\label{stima}
| b(t,\cdot)-b_0(t,\cdot)|_{k}\lesssim \la z\ra^{-2},\quad t\in[0,T],
\end{equation}	
where $b_0(t,z)=e^{-t a_{\beta,0}(z)}$, and the semi-norms $|\cdot |_k$, $k\in \bN$,  are defined by 
\begin{equation}\label{semingamm}
|a|_k:=\sup_{|\alpha|+|\beta|\leq k}|\partial^\alpha_\xi \partial^\beta_x a(x,\xi)|\la (x,\xi)\ra^{|\alpha|+|\beta|}.
\end{equation}
Defining the remainder $R_1(t,z):= b(t,z)-b_0(t,z)$, we infer from \eqref{stima} that $R_1(t,\cdot)\in \Gamma_1^{-2}(\rdd)\subset  \Gamma_1^{0}(\rdd) $ uniformly w.r.t. $t$ on $[0,T]$.

Moreover, using the ellipticity condition $a_{\beta,0}(z)\gtrsim\la z \ra,$ for $ |z|\geq R$,  and the property  $a_{\beta,0}\in \Gamma_1^{2\beta}(\rdd)$,
one easily shows by induction that there exists a  constant $C>0:$ 
$$|\partial^\alpha_z b_0(t,z)|\leq C \la z\ra^{-|\alpha|},\quad \forall  t\in[0,T]$$
that is to say, the symbol $b_0(t,z)$ is in the Shubin class $\Gamma^0_1(\rdd)$
with uniform estimates w.r.t. the time variable $t\in [0,T]$.
Hence,  $b(t,\cdot)=b_0(t,\cdot)+R_1(t,\cdot)\in \Gamma^0_1(\rdd)$, with uniform estimate w.r.t. $t\in [0,T]$.\par  
 For applications  of Shubin classes in the framework of Born-Jordan quantization we refer to the work \cite{cgnp}.

\subsection{Gabor analysis of $\tau$-pseudodifferential operators} 	
For $\tau \in [0,1]$, $f,g\in\lrd$, the (cross-)$\tau $-Wigner distribution is defined  by
\begin{equation}
W_{\tau }(f,g)(x,\omega )=\int_{\mathbb{R}^{d}}e^{-2\pi iy\omega }f(x+\tau y)%
\overline{g(x-(1-\tau )y)}\,dy,\quad x,\omega\in\rd.
\label{tauwig}
\end{equation}
It can be used to define the \tpsdo\ with symbol $\sigma$ via the formula
\begin{equation}
\langle \opt(\sigma)f,g\rangle =
\langle
\sigma,W_{\tau }(g,f)\rangle, \quad f,g\in \mathcal{S}(\mathbb{R}^{d}).
\label{tauweak}
\end{equation}
For $\tau=1/2$  we recapture the Weyl operator.  We want to consider \tpsdo s  with symbols $\sigma$ in the H\"{o}rmander class $S^m_{0,0}$, $m\in\bR$, consisting of functions $\sigma\in \cC^\infty(\rdd)$ such that, for every $\a\in \bN^{2d}$,
\begin{equation}\label{semi-normSm}
|\partial^\alpha \sigma(z)|\leq C_\alpha \la z\ra^m, \quad z\in \rdd.
\end{equation}
The related semi-norms  are denote by 
\begin{equation}\label{semihormander}
|\sigma|_{N,m}:=\sup_{|\a|\leq N} |\partial^\alpha \sigma(z)|\la z\ra^{-m}. 
\end{equation}
Fix $g\in \cS(\rd)\setminus \{0\}$. We define the \emph{Gabor matrix} of a linear continuous operator $T$ from $\cS(\rd)$ to $\cS'(\rd)$, the  mapping from  $\rdd\times\rdd$ into $\bC$,
\begin{equation}\label{unobis2s} (z,y)\mapsto \langle T \pi(z)
g,\pi(y)g\rangle,\quad z,y\in \rdd.
\end{equation}
This is a slightly abuse of notation, since originally Gabor matrices we defined for \tfs s $\pi(\lambda)$, with $\lambda$ varying in a lattice $\Lambda\subset \rdd$. We observe that the almost diagonalization of Gabor matrices of pseudodifferential operators with symbols in the modulation space $M^{\infty,1}(\rdd)$ treated in \cite{grochenig2} (and in many subsequent papers on the topic)  are valid in both the continuous and discrete case. So we adopt this terminology in the continuous framework. 

 For $m=0$ we are reduced to the  H\"{o}rmander class $S^0_{0,0}$, whose Gabor matrix  characterization for Weyl operators was shown in \cite[Theorem 6.1]{GR}, see also \cite{rochberg}. Even though $m=0$ is our case of interest,  for our goal we need to control such matrix by the semi-norms of  $S^0_{0,0}$. Moreover, for further references, we shall formulate our result in the case of  \tpsdo s having symbols in the more general class $S^m_{0,0}$, $m\in\bR$.

We are going to use the following result  for $\tau$-pseudodifferential operators  \cite[Lemma 4.1]{CNT}.
\begin{lemma}
	\label{lem:STFT-gaborm} Fix a non-zero window $g\in \cS(\rd)$
	and set $\Phi_{\tau}=W_{\tau}(g,g)$ for $\tau\in\left[0,1\right]$.
	Then, for $\sigma\in \cS'\left(\mathbb{R}^{2d}\right)$,
	\begin{equation}
\left|\left\langle \opt \left(\sigma\right)\pi\left(z\right)g,\pi\left(y\right)g\right\rangle \right|=\left|{V}_{\Phi_{\tau}}\sigma\left(\mathcal{T}_{\tau}\left(z,y\right),J\left(y-z\right)\right)\right|\label{eq:gaborm as STFT}.
\end{equation}
where $z=(z_1,z_2)$, $y=(y_1,y_2)$ and $\mathcal{T}_{\tau}$ and $J$ are defined as follows:
\begin{equation*}
\mathcal{T}_{\tau}(z,y)=((1-\tau)z_1+\tau y_1,\tau z_2+(1-\tau)y_2),\quad J(z)=(z_2,-z_1).
\end{equation*}
\end{lemma}

The Gabor matrix for a  \tpsdo\, $\opt(\sigma)$ with symbol $\sigma\in S^m_{0,0}$ enjoys the following decay.
\begin{theorem}\label{GBsm}
Fix $g\in\cS(\rd)\setminus\{0\}$,  $m\in\bR$, $\tau\in [0,1]$. Consider a  \tpsdo\, $\opt(\sigma)$ with symbol $\sigma\in S^m_{0,0}$. Then for every $N\in\bN$ there exists $C=C(N)>0$ such that
\begin{equation}\label{ltau}
\left|\left\langle \opt \left(\sigma\right)\pi\left(z\right)g,\pi\left(y\right)g\right\rangle \right|\leq C |\sigma|_{2N,m} \frac{\la \mathcal{T}_{\tau}(z,y)\ra ^m }{\la y-z\ra^{2N}},\quad z,y\in\rdd,
\end{equation}
where the semi-norms $|\cdot|_{N,m}$ are defined in \eqref{semihormander}.
\end{theorem}
\begin{proof}
Using the representation 	in \eqref{eq:gaborm as STFT} and 
$$ (1-\Delta_\lambda)^N e^{-2\pi i \lambda J(y-z)} =\la 2\pi(y-z)\ra ^{2N} e^{-2\pi i \lambda J(y-z)}$$ we can write
\begin{align*}
\left|\left\langle \opt \left(\sigma\right)\pi\left(z\right)g,\pi\left(y\right)g\right\rangle \right|&\leq C \frac{1}{\la 2\pi(y-z)\ra ^{2N}}\\
&\quad \times\quad \left|\intrdd e^{-2\pi i \lambda J(y-z)} (1-\Delta_\lambda)^N \left[\bar{\sigma}(\lambda)T_{ \mathcal{T}_{\tau}(z,y)}\Phi_\tau\right] d\lambda\right|
\end{align*}
(observe that the above integral is absolutely convergent since $\Phi_\tau\in\cS(\rdd)$). Now we estimate
\begin{align*}(1-\Delta_\lambda)^N \left[\bar{\sigma}(\lambda)T_{ \mathcal{T}_{\tau}(z,y)}\Phi_\tau\right]&\leq \sum_{|\a|+|\beta|\leq 2N}|C_{\a,\beta}| |\partial^\alpha \sigma(\lambda)| \,|\partial^\beta \Phi_\tau(\lambda-\mathcal{T}_{\tau}(z,y)))|\\
&\leq C_N |\sigma|_{2N,m}  \la \lambda\ra ^m \la \lambda-\mathcal{T}_{\tau}(z,y)\ra^{-s}
\end{align*}
for every $s\geq 0$ since $\Phi_\tau\in\cS(\rdd)$. Choose $s=|m|+2d+1$. Then
the submultiplicativity of $\la \cdot\ra ^{|m|}$ allows us to control from above the right-hand side of the last inequality by 
$$ C_N |\sigma|_{2N,m}  \la \mathcal{T}_{\tau}(z,y)\ra ^m \la \lambda -\mathcal{T}_{\tau}(z,y)\ra^{-(2d+1)}.$$
Hence, for every $N\in\bN$ we can find $C(N)>0$ such that \eqref{ltau} is satisfied. This concludes the proof.
\end{proof}

For fixed $\tau\in(0,1)$ we observe that $\la \mathcal{T}_{\tau}(z,y)\ra\asymp \la z+y\ra$, hence the matrix decay can be controlled by a function which does not depend on the $\tau$-quantization. Namely,
\begin{corollary}
	Fix $g\in\cS(\rd)\setminus\{0\}$,  $m\in\bR$, $\tau\in (0,1)$. Consider a  \tpsdo\, $\opt(\sigma)$ with symbol $\sigma\in S^m_{0,0}$. Then, for every $N\in\bN$, there exists $C=C(\tau,N)>0$ such that
	\begin{equation}\label{ltau}
	\left|\left\langle \opt \left(\sigma\right)\pi\left(z\right)g,\pi\left(y\right)g\right\rangle \right|\leq C |\sigma|_{2N,m} \frac{\la z+y\ra ^m }{\la y-z\ra^{2N}},\quad z,y\in\rdd.
	\end{equation}
	\end{corollary}
\begin{Remark}
	We conjecture that pseudodifferential  operators in the H\"{o}rmander class $S^m_{0,0}$, $m\in\bR$, can be characterized via the Gabor matrix in \eqref{ltau}, extending the case $m=0$ already shown in \cite{GR}. To be precise, we allow to write
	$$S^m_{0,0}=\bigcap_{s\geq 0} M^\infty_{\la \cdot\ra ^m\otimes \la \cdot\ra^s}=\bigcap_{s\geq 0} M^{\infty,1}_{\la \cdot\ra ^m\otimes \la \cdot\ra^s}.$$
Studying  the Gabor matrix decay for $M^\infty_{\la \cdot\ra ^m\otimes \la \cdot\ra^s}$ and following the pattern of the proofs as in the paper  \cite{GR} one should get the result  easily. Since this subject is outside the scope of the paper, we will write the details in a separate work.
\end{Remark}


%

\section{Local  Well-posedness  in modulation spaces}

For $m=0$ the semi-norms on $\Gamma^0_1$ are exactly the ones in \eqref{semingamm}.  Observe that 
\begin{equation}\label{inclsem}
\Gamma^m_1\hookrightarrow S^m_{0,0}
\end{equation}
(the inclusion is continuous).  

 The results in the previous yields the boundedness of the Weyl operator $ e^{-t H^\beta}$ on modulation spaces.
\begin{theorem}\label{T3.1}
Consider $1\leq p,q\leq\infty$, $0<\beta\leq 1$, $w\in \mathcal{M}_{v}$. Then for every $T>0$ there exists $C=C(T)>0$ such that
\begin{equation}\label{20}
\| e^{-t H^\beta} u_0\|_{{M}_w^{p,q}}\leq C \|u_0\|_{{M}_w^{p,q}}, \quad \forall t\in [0,T], \quad u_0\in  {M}_w^{p,q}(\rd).
\end{equation}
\end{theorem}
\begin{proof} Consider $u_0\in {M}_w^{p,q}(\rd)$.
	Fix $g\in\cS(\rd)\setminus\{0\}$ such that $\|g\|_2=1$. Then using the inversion formula for $u_0$ in \eqref{invformula} we can write 
\begin{align*}
|V_g(e^{-t H^\beta }u_0)(y)w(y)|&=\left|\intrdd \la e^{-t H^\beta} \pi (z)g, \pi(y)g\ra V_g u_0(z)\,dz\right|\\
&\leq \intrdd w(y) |\la e^{-t H^\beta} \pi (z)g, \pi(y)g\ra| |V_g u_0(z)|\,dz
\end{align*}
In the previous section we showed that $e^{-t H^\beta}$ is a Weyl operator with symbol $b(t,\cdot)$ in $\Gamma_1^0$ with semi-norms uniformly bounded w.r.t. $t\in [0,T].$ The continuous embedding in \eqref{inclsem} and Theorem \ref{GBsm} let us write
$$|\la e^{-t H^\beta} \pi (z)g, \pi(y)g\ra|\leq C  \frac{1}{\la y-z\ra^{2N}}.$$
Since $w(y)\lesssim v(y-z)  w(z)$, and $v(z)\lesssim \la z\ra ^s$ for some $s>0$, we can write 
\begin{align*}
|V_g(e^{-t H^\beta}u_0)(y)w(y)|&\leq  C \intrdd \la y-z\ra^s  w(z) |V_g u_0(z)\frac{1}{\la y-z\ra^{2N}}dz\\
&\leq C \left[\frac{1}{\la \cdot\ra^{2N-s}}\ast (|V_g u_0|w)\right](y). 
\end{align*}
Choosing $N$ such that $2N-s>2d+1$ and using the convolution relations $L^1\ast L^{p,q}\hookrightarrow L^{p,q}$ we obtain the claim.
\end{proof}

Choosing $p=q=2$ and recalling that, for $w(x,\xi)=\la \xi\ra^s$, $M^2_w(\rd)=H^s(\rd)$ (Sobolev spaces), whereas for $w(z)=\la z\ra^s$, $z\in\rdd$, $M^2_w(\rd)=\mathcal{Q}_s$ (Shubin-Sobolev spaces), cf., e.g., \cite[Chapter 2]{CR}, we obtain boundedness results also for these classical spaces.

\begin{corollary} Consider $0<\beta\leq 1$, $s\in\bR$. For any fixed $T>0$ there exists 
$C=C(T)>0$ such that 
\begin{equation}\label{201}
\| e^{-t H^\beta} u_0\|_{H^s}\leq C \|u_0\|_{H^s}, \quad \forall t\in [0,T],\quad u_0\in H^s(\rd).
\end{equation}
The same result holds by replacing the Sobolev space $H^s$ with the Shubin-Sobolev space $\mathcal{Q}_s$.
\end{corollary}

As already done in \cite{CZ}, in order to show the local existence of the solution we will make use of the following variant of the contraction mapping theorem (cf., e.g., \cite[Proposition
1.38]{tao}).
\begin{proposition}\label{AIA}
	Let $\cN$ and $\cT$ be two
	Banach spaces. Consider a linear operator
	$\cB:\cN\to \cT$ such that 
	\begin{equation}\label{aia1}
	\|\cB f\|_{\cT}\leq
	C_0\|f\|_{\cN},\quad \forall f\in\cN,
	\end{equation}
	for some
	$C_0>0$, and suppose to have
	 a nonlinear
	operator $F:\cT\to\cN$ with
	$F(0)=0$ and Lipschitz bounds
	\begin{equation}\label{aia2}
	\|F(u)-F(v)\|_{\cN}\leq\frac{1}{2C_0}\|u-v\|_{\cT},
	\end{equation}
	for all $u,v$ in the ball
	$B_\mu:=\{u\in\cT:
	\|u\|_\cT\leq \mu \}$, for
	some $\mu>0$. Then, for all
	$u_{\rm lin}\in B_{\mu/2}$
	there exists a unique
	solution $u\in B_\mu$ to the
	equation $
	u=u_{\rm lin}+\cB F(u),
	$
	with the map $u_{lin}\mapsto
	u$ Lipschitz continuous with constant at
	most $2$.
\end{proposition}

%
%

\begin{proof}[Proof of Theorem \ref{T1}]
We apply Theorem \ref{T3.1} with $T=1$, $q=1$ and $w(x,\xi)=\la \xi\ra^s$. For every
$1\leq p< \infty$, the
operator $K_\beta(t)$ in \eqref{op2} is a bounded
operator on
${M}^{p,1}_s(\rd)$, and there exists a $C>0$ such that 
\begin{equation}\label{G1} \|K_\beta(t)
u_0\|_{{M}^{p,1}_s}\leq
C \|u_0\|_{{M}^{p,1}_s},\quad
t\in [0,1].
\end{equation}
Notice that such result provides the uniformity of the constant C, when t varies in $[0,1]$.
 Now  the result follows by Proposition \ref{AIA}, with
$\cT=\cN=C^0([0,T];{M}^{p,1}_s)$, the
linear operator $\cB$ in \eqref{op2},  where  $0<T\leq 1$ will be
chosen  later on. Here $u_{\rm
lin}:=K_\beta(t)u_0$ is in the ball $B_{\mu/2}\subset\cT$ by
\eqref{G1},  if $\mu$ is sufficiently large, depending
on the radius  $R$ in $M^{p,1}_s(\rd)$ in the assumptions. Using Minkowski's integral inequality and \eqref{G1}, we
obtain \eqref{aia1}. Namely,
$$\|\cB u\|_{{M}^{p,1}_s}\leq T C\|u\|_{_{{M}^{p,1}_{s}}}.
$$
\par The proof of  Condition \eqref{aia2} can be found   in \cite[proof of Theorem 4.1]{CNwave}. Hence, by choosing $T$
small enough we
prove the 
existence, and also the
uniqueness among the solution
in $\cT$ with norm $O(R)$  (with $R$ being the radius of the ball $B_R$, centred in $0$, in $M^{p,1}_s(\rd)$).
Standard continuity arguments allow to eliminate the last constraint
(see, e.g., \cite[Proposition 3.8]{tao}).
For $p=\infty$, by repeating the argument above, one can obtain well-posedness when the initial datum is in    $$\mathcal{M}^{\infty,1}_s(\rd):=\overline{\cS}^{M^{\infty,1}_s}(\rd).$$
\end{proof}

Observe that similar results were obtained in \cite[Theorem 1.1]{nicola}.

We conclude this   note by addressing the reader  to open problems in this field.

First, it is still not clear whether better results  can be obtained when considering  the nonlinearity
\begin{equation} \label{PW}
F(u)=F_k(u)=\lambda
|u|^{2k}u=\lambda
u^{k+1}\bar{u}^k, \quad
\lambda\in\mathbb{C},\
k\in\N.
\end{equation}
In fact, this was the case for the wave and vibrating place equation, cf. \cite{CNwave,CZ}, where more general modulation spaces were considered.

Moreover, another open question is the well-posedness  of the Cauchy problem \eqref{cpw} with initial datum $u_0\in M^{p,q}_s(\rd)$,  $0<p\leq \infty$, $0<q\leq1$. We conjecture that the result holds true as well, but the techniques employed so far do not apply in this case.

\section*{Acknowledgements}
The author would like to thank Professors Fabio
Nicola and Luigi Rodino for fruitful conversations and comments.

\end{document}